\documentclass[11pt, a4paper]{amsart}
\usepackage{fullpage}
\usepackage{amssymb}
\usepackage{commath}
\usepackage{empheq}
\usepackage{amsmath, amsthm}
\usepackage[foot]{amsaddr}
\usepackage{dsfont}
\usepackage{multicol}
\usepackage{mathtools}
\usepackage{pdflscape}

\usepackage{etoolbox}
\makeatletter
\let\ams@starttoc\@starttoc
\makeatother
\usepackage[parfill]{parskip}
\makeatletter
\let\@starttoc\ams@starttoc
\patchcmd{\@starttoc}{\makeatletter}{\makeatletter\parskip\z@}{}{}
\makeatother

\usepackage{graphicx}

\usepackage{pgfplots}
\usepackage{tikz}
\usetikzlibrary{matrix,arrows,decorations.pathmorphing,decorations.markings}
\usepackage{tikz-cd}

\usepackage{tikz}
\usetikzlibrary{backgrounds}
\usetikzlibrary{arrows}
\usetikzlibrary{shapes,shapes.geometric,shapes.misc}

\pgfkeys{/tikz/tikzit fill/.initial=0}
\pgfkeys{/tikz/tikzit draw/.initial=0}
\pgfkeys{/tikz/tikzit shape/.initial=0}
\pgfkeys{/tikz/tikzit category/.initial=0}

\pgfdeclarelayer{edgelayer}
\pgfdeclarelayer{nodelayer}
\pgfsetlayers{background,edgelayer,nodelayer,main}

\tikzstyle{none}=[inner sep=0mm]

\usepackage{capt-of}

\usepackage[titletoc]{appendix}

\usepackage{enumitem}

\usepackage{mathrsfs}

\usepackage{url}

\usepackage{faktor}
\usepackage{verbatim}

\newcommand*{\QEDA}{\hfill\ensuremath{\square}}
\newcommand{\mr}{\mathrm}
\newcommand{\ms}{\mathscr}
\newcommand{\p}{\partial}
\newcommand{\SB}{B_A^{sym}}

\newcommand{\ul}{\underline}
\newcommand{\ol}{\overline}

\newcommand{\xra}{\xrightarrow}

\newcommand{\G}{\Gamma}
\newcommand{\DS}{\Delta S}
\newcommand{\EDS}{\mathrm{Epi}\Delta S}
\newcommand{\SBI}{B_I^{sym}}

\newcommand{\ve}{\varepsilon}

\renewcommand{\phi}{\varphi}

\newcommand{\fas}{\mathcal{F}(as)}

\newcommand{\llb}{\left\lbrace}
\newcommand{\rrb}{\right\rbrace}
\newcommand{\llv}{\left\lvert}
\newcommand{\rrv}{\right\rvert}

\newcommand{\lodam}{\ms{L}(A,M)}

\newcommand{\llangle}{\left\langle}
\newcommand{\rrangle}{\right\rangle}

\newcommand{\RNum}[1]{\MakeUppercase{\romannumeral #1}}

\renewcommand{\phi}{\varphi}

\newcommand{\CM}{\mathbf{CMod}}
\newcommand{\MC}{\mathbf{ModC}}

\newcommand{\kmod}{\mathbf{kMod}}

\newcommand{\C}{\mathbf{C}}

\newtheorem{thm}{Theorem}[section]

\theoremstyle{plain}
\newtheorem{prop}[thm]{Proposition}
\newtheorem{lem}[thm]{Lemma}
\newtheorem{cor}[thm]{Corollary}

\theoremstyle{definition}
\newtheorem{defn}[thm]{Definition}

\makeatletter
\@addtoreset{step}{proof}
\makeatother
\tikzset{commutative diagrams/.cd,arrow style=tikz,diagrams={>=latex'}}

\theoremstyle{remark}
\newtheorem{rem}[thm]{Remark}

\usepackage[all]{xy}

\title{A Comparison Map for Symmetric Homology and Gamma Homology}
\author{Daniel Graves}
\address{School of Mathematics and Statistics, University of Sheffield, Sheffield, S3 7RH, UK}
\date{}

\begin{document}


\maketitle

\begin{abstract}
We construct a comparison map between the gamma homology theory of Robinson and Whitehouse and the symmetric homology theory of Fiedorowicz and Ault in the case of an augmented, commutative algebra over a unital commutative ground ring.
\end{abstract}

\section*{Introduction}
The gamma homology theory of Robinson and Whitehouse \cite{Whi94} and the symmetric homology theory of Fiedorowicz, \cite{Fie} and \cite{Ault}, are both constructed by building a symmetric group action into Hochschild homology, albeit in very different ways. In this paper we provide a comparison map for the two theories in the case of an augmented, commutative algebra.

Robinson and Whitehouse developed \emph{gamma homology}, which we will frequently write as $\G$-homology, for commutative algebras to encode information about homotopy commutativity. $\G$-homology is closely related to stable homotopy theory as demonstrated by Pirashvili \cite{PirH} and Pirashvili and Richter \cite{Pir00}. When the ground ring contains $\mathbb{Q}$, $\G$-homology coincides with Harrison homology up to a shift in degree. We will define a normalized version of Harrison homology for an augmented, commutative algebra.

The \emph{symmetric homology} theory for associative algebras was first introduced by Fiedorowicz \cite{Fie} and was extensively developed by Ault, \cite{Ault} and \cite{Ault-HO}. It is the homology theory associated to the symmetric crossed simplicial group \cite{FL}. Symmetric homology also has connections to stable homotopy theory. For instance, the symmetric homology of a group algebra is isomorphic to the homology of the loop space on the infinite loop space associated to the classifying space of the group \cite[Corollary 40]{Ault}.

Defining a comparison map between the two homology theories is not straightforward in general. Our results utilize the fact in the case of an augmented, commutative algebra there exist smaller chain complexes than the standard complexes used to compute each homology theory.

The paper is arranged as follows. Section \ref{background-section} collates some prerequisites on symmetric groups, augmented algebras and chain complexes associated to simplicial $k$-modules. We recall the Hochschild complex, the shuffle subcomplex and the Eulerian idempotents as these will play an important role in defining normalized Harrison homology. Section \ref{FH-section} recalls some constructions of functor homology. Symmetric homology is defined in terms of functor homology and $\G$-homology can be similarly defined in this setting. In particular, we recall the notion of $\mr{Tor}$ functors for a small category and a chain complex that computes them.

We begin Section \ref{Harr-section} by recalling the definition of \emph{Harrison homology} for commutative algebras. We construct a normalized version of Harrison homology for augmented, commutative algebras over a ground ring containing $\mathbb{Q}$. In particular, we prove that under these conditions there is a smaller chain complex, constructed from the augmentation ideal, that computes Harrison homology.

In Section \ref{gamma-section} we recall the defintion of $\G$-homology for a commutative algebra. We recall the \emph{Robinson-Whitehouse complex}, whose homology is $\G$-homology. For an augmented, commutative algebra we prove that the Robinson-Whitehouse complex splits using a construction called the \emph{pruning map}. One summand of this splitting is constructed analogously to the Robinson-Whitehouse complex using the augmentation ideal. We prove that for a flat algebra over a ground ring containing $\mathbb{Q}$, the homology of this summand is isomorphic to the $\G$-homology of the algebra. 

Section \ref{symm-hom-section} recalls the definition of symmetric homology. We recall Ault's chain complex for computing the symmetric homology of an augmented algebra.

In Section \ref{CM-section} we construct a surjective map of chain complexes which gives rise to a long exact sequence connecting the symmetric homology of an augmented commutative algebra with a direct summand of its gamma homology. We use the material of Sections \ref{Harr-section} and \ref{gamma-section} to prove that when the algebra is flat over a ground ring containing $\mathbb{Q}$ we have a long exact sequence connecting the symmetric homology of the algebra with its gamma homology.

The results in this paper first appeared in the author's thesis \cite{DMG}. We note that the statements of Theorem \ref{gamma-iso-thm} and Theorem \ref{long-exact-thm} in this paper have an extra flatness condition that was accidentally omitted in \cite[Theorem 32.2.1]{DMG} and \cite[Theorem 32.2.2]{DMG}. 

Throughout the paper, $k$ will denote a unital commutative ring.

\section{Background}
\label{background-section}
We collect prerequisites for the remainder of the paper. We will require facts about the symmetric groups for each of Harrison homology, $\G$-homology and symmetric homology. We recall some facts about augmented algebras that will be required for our results. We briefly recall the construction of the Hochschild complex, the shuffle subcomplex and the Eulerian idempotents, all of which can be found in \cite{Lod}, as these will be essential in defining a normalized Harrison homology in Section \ref{Harr-section}.

\subsection{Symmetric groups and shuffles}

\begin{defn}
The symmetric group $\Sigma_{n}$ on $\ul{n}=\lbrace 1,\dotsc ,n\rbrace$ for $n\geqslant 1$ is generated by the transpositions $\theta_i= (i\,\, i+1)$ for $1\leqslant i\leqslant n-1$ such that each transposition squares to the identity, disjoint transpositions commute and $\left(\theta_i\circ\theta_{i+1}\right)^3=id_{\ul{n}}$ for $1\leqslant i\leqslant n-2$.
\end{defn}

\begin{defn}
For $1\leqslant i\leqslant n-1$, a permutation $\sigma \in \Sigma_n$ is called an $i$-\emph{shuffle} if
\[ \sigma(1)<\sigma(2)<\cdots < \sigma(i) \quad \text{and} \quad \sigma(i+1)<\sigma(i+2)<\cdots <\sigma(n).\]
\end{defn}

\begin{defn}
We define an element $sh_{i,n-i}$ in the group algebra $k[\Sigma_n]$ by
\[sh_{i,n-i}=\sum_{\substack{i-\text{shuffles}\\ \sigma\in \Sigma_n}} sgn(\sigma)\sigma.\]
That is, we take the signed sum over all the $i$-shuffles in $\Sigma_n$.
\end{defn}

\begin{defn}
\label{tot-shuff}
We define the \emph{total shuffle operator}, $sh_n$, in $k[\Sigma_n]$ to be the sum of the elements $sh_{i,n-i}$. That is,
\[sh_n=\sum_{i=1}^{n-1} sh_{i,n-i}.\]
\end{defn}

\subsection{Augmented algebras}
\begin{defn}
A $k$-algebra $A$ is said to be \emph{augmented} if it is equipped with a $k$-algebra homomorphism $\varepsilon\colon A\rightarrow k$. We define the \emph{augmentation ideal} of $A$ to be $\mathrm{Ker} (\ve)$ and we denote it by $I$.
\end{defn}

Recall from \cite[Section 1.1.1]{LodVal} that for an augmented $k$-algebra $A$ there is an isomorphism of $k$-modules $A\cong I\oplus k$. It follows that every element $a\in A$ can be written uniquely in the form $y+\lambda$ where $y\in I$ and $\lambda \in k$.

\begin{defn}
\label{basic-tens-defn}
Let $A$ be an augmented $k$-algebra with augmentation ideal $I$. A \emph{basic tensor} in $A^{\otimes n}$ is an elementary tensor $a_1\otimes \cdots \otimes a_n$ such that either $a_i\in I$ or $a_i=1_k$ for each $1\leqslant i\leqslant n$. A tensor factor $a_i$ is called \emph{trivial} if $a_i=1_k$ and is called \emph{non-trivial} if $a_i\in I$.
\end{defn}
One notes that the $k$-module $A^{\otimes n}$ is generated $k$-linearly by all basic tensors.

\begin{defn}
A \emph{bimodule} over an associative $k$-algebra $A$ is a $k$-module $M$ with a $k$-linear action of $A$ on the left and right satisfying $(am)a^{\prime}=a(ma^{\prime})$ where $a$, $a^{\prime}\in A$ and $m\in M$. An $A$-bimodule $M$ is said to be \emph{symmetric} if $am=ma$ for all $a\in A$ and $m\in M$.
\end{defn}

One notes that for an augmented $k$-algebra $A$, $k$ is an $A$-bimodule with the structure maps determined by the augmentation.

\subsection{Hochschild homology}
\begin{defn}
\label{complexes}
Let $X_{\star}$ be a simplicial $k$-module. The \emph{associated chain complex}, denoted $C_{\star}(X)$ is defined to have the $k$-module $X_n$ in degree $n$ with the boundary map defined to be the alternating sum of the face maps of $X_{\star}$.

The \emph{degenerate subcomplex}, denoted $D_{\star}(X)$, is defined to be the subcomplex of $C_{\star}(X)$ generated by the degeneracy maps of $X_{\star}$.

The \emph{normalized chain complex}, denoted $N_{\star}(X)$, is defined to be the quotient of $C_{\star}(X)$ by the subcomplex $D_{\star}(X)$.
\end{defn}

\begin{rem}
The normalized chain complex is isomorphic to the Moore subcomplex of $C_{\star}(X)$ \cite[Definition 8.3.6]{weib}.
\end{rem}

\begin{defn}
\label{HH-simp-defn}
Let $A$ be an associative $k$-algebra and let $M$ be an $A$-bimodule. The \emph{Hochschild complex}, denoted $C_{\star}(A,M)$, is the associated chain complex of the simplicial $k$-module with $C_n(A,M)= M\otimes A^{\otimes n}$, the $k$-module generated $k$-linearly by all elementary tensors $(m\otimes a_1\otimes\cdots \otimes a_n)$. The face maps $\p_i\colon C_n(A,M)\rightarrow C_{n-1}(A,M)$ are determined by
\begin{itemize}
\item $\p_0(m\otimes a_1\otimes\cdots \otimes a_n)=(ma_1\otimes a_2\otimes\cdots \otimes a_n)$,
\item $\p_i(m\otimes a_1\otimes \cdots \otimes a_n)=(m\otimes a_1\otimes \cdots \otimes a_ia_{i+1}\otimes\cdots \otimes a_n)$ for $1\leqslant i\leqslant n-1$,
\item $\p_n(m\otimes a_1\otimes\cdots \otimes a_n)=(a_nm\otimes a_1\otimes\cdots \otimes a_{n-1})$
\end{itemize}
and the degeneracy maps insert the multiplicative identity $1_A$ into the elementary tensor. The boundary map is denoted by $b$. In the case where $M=A$ we denote the Hochschild complex by $C_{\star}(A)$. The homology of $C_{\star}\left(A,M\right)$ is denoted by $HH_{\star}\left(A,M\right)$ and is called the \emph{Hochschild homology} of $A$ with coefficients in $M$.
\end{defn}

\begin{prop}
\label{AugSplit}
Let $A$ be an augmented $k$-algebra with augmentation ideal $I$. Let $M$ be an $A$-bimodule which is flat over $k$. The normalized complex of the Hochschild complex is canonically isomorphic to the chain complex $C_{\star}(I,M)$, formed analogously to the Hochschild complex.  
\end{prop}
\begin{proof}
This result is well-known. A full proof can be found in \cite[Section 5.2]{DMG}.
\end{proof}

\subsection{Shuffle complex}
It is shown in \cite[Section 4.2.8]{Lod} that the total shuffle operators of Definition \ref{tot-shuff} form a chain map $sh_n\colon C_n(A)\rightarrow C_n(A)$, allowing us to make the following definition.

\begin{defn}
\label{shuff-cpx-defn}
Let $a_0 \otimes a_1\otimes\cdots \otimes a_n$ be an elementary tensor of the $k$-module $C_n(A)$. Let \[\sum_{i=1}^{n-1}sh_{i,n-i}(a_0\otimes a_1\otimes\cdots \otimes a_n)\] denote the linear combination of tensors obtained by applying the total shuffle operator. Let $Sh_n(A)$ denote the submodule of $C_n(A)$ generated by all such $k$-linear combinations obtained from the $k$-module generators of $C_n(A)$. 
The \emph{shuffle complex}, denoted $Sh_{\star}(A)$, is defined to have the module $Sh_n(A)$ in degree $n$ with boundary map induced from the Hochschild differential.
\end{defn}

\subsection{Eulerian idempotents}
In the case where the ground ring $k$ contains $\mathbb{Q}$ and $A$ is a commutative $k$-algebra, we can decompose the Hochschild complex into a direct sum of subcomplexes. We recall the Eulerian idempotents from \cite[Proposition 4.5.3]{Lod}.

\begin{prop}
\label{idempotent}
For $n\geqslant 1$ and $1\leqslant i\leqslant n$, there exist pairwise-disjoint idempotent elements $e_n^{(i)}\in \mathbb{Q}[\Sigma_n]$ such that $e_n^{(1)}+\cdots +e_n^{(n)}$ is the identity in $\mathbb{Q}[\Sigma_n]$. These elements are called the \emph{Eulerian idempotents}. \QEDA
\end{prop}

The Eulerian idempotents act on $C_n(A,M)=M\otimes A^{\otimes n}$ by $1_M\otimes e_n^{(i)}$. Let $(m\otimes a_1\otimes\cdots \otimes a_n)$ be a $k$-module generator of $C_n(A,M)$. Let $e_n^{(i)}(m\otimes a_1\otimes\cdots\otimes a_n)$ denote the $k$-linear combination of tuples obtained by applying the Eulerian idempotent to our generator. This is simply a linear combination of permutations of our generator. The subcomplex $e_{\star}^{(i)}C_{\star}(A,M)$ is generated in degree $n$ by all linear combinations of the form $e_n^{(i)}(m\otimes a_1\otimes\cdots\otimes a_n)$, arising from $k$-module generators of $C_n(A,M)$. The following is a theorem of Gerstenhaber and Schack \cite{GS87}. 

\begin{thm}
\label{Hodge}
Let $k$ be a commutative ring containing $\mathbb{Q}$. Let $A$ be a $k$-algebra and let $M$ be a symmetric $A$-bimodule. The Eulerian idempotents $e_n^{(i)}$ naturally split the Hochschild complex into a direct sum of subcomplexes:
\[C_{\star}(A,M)=\bigoplus_{i= 1}^{\infty} e_{\star}^{(i)}C_{\star}(A,M).\]
Thus, this provides a decomposition of Hochschild homology:
\[\pushQED{\qed} H_n(A,M)= H_n\left(e_n^{(1)}C_{\star}(A,M)\right)\oplus \cdots \oplus H_n\left(e_n^{(n)}C_{\star}(A,M)\right). \qedhere \popQED\]
\end{thm}

\section{Functor Homology}
\label{FH-section}
Symmetric homology \cite[Definition 14]{Ault} is defined in the setting of functor homology. Gamma homology has a functor homology description due to Pirashvili and Richter \cite{Pir00}. In this section we recall the necessary constructions of functor homology, as found in \cite{Pir02} for example. We also recall a chain complex construction of Gabriel and Zisman \cite[Appendix 2]{GZ} used to calculate the homology groups of functors. 

\subsection{Modules over a category}
\begin{defn}
Let $\C$ be a small category. We define the \emph{category of left $\mathbf{C}$-modules}, denoted $\CM$, to be the functor category $\mr{Fun}\left(\mathbf{C},\kmod\right)$. We define the \emph{category of right $\mathbf{C}$-modules}, denoted $\MC$, to be the functor category $\mr{Fun}\left(\mathbf{C}^{op},\kmod\right)$.
\end{defn}

\begin{defn}
\label{triv-c-mod}
We define the \emph{trivial right $\mathbf{C}$-module}, $k^{\star}$, to be the constant functor at the trivial $k$-module.
\end{defn}

It is well-known that the categories $\CM$ and $\MC$ are abelian, see for example \cite[Section 1.6]{Pir02}.

\subsection{Tensor product of functors}
\begin{defn}
\label{tensor-obj}
Let $G$ be an object of $\MC$ and $F$ be an object of $\CM$. We define the tensor product $G\otimes_{\C} F$ to be the $k$-module
\[\frac{\bigoplus_{C\in \mr{Ob}(\C)} G(C)\otimes_k F(C)}{\llangle G(\alpha)(x)\otimes y - x\otimes F(\alpha)(y)\rrangle}\]
where 
\[\llangle G(\alpha)(x)\otimes y - x\otimes F(\alpha)(y)\rrangle\]
is the $k$-submodule generated by the set
\[\llb G(\alpha)(x)\otimes y - x\otimes F(\alpha)(y): \alpha \in \mr{Hom}(\C),\,\, x\in src(G(\alpha)),\,\, y\in src(F(\alpha))\rrb.\]
\end{defn}

This quotient module is spanned $k$-linearly by equivalence classes of elementary tensors in $\bigoplus_{C\in \mr{Ob}(\C)} G(C)\otimes_k F(C)$ which we will denote by $\left[x\otimes y\right]$.

\begin{defn}
The bifunctor
\[-\otimes_{\mathbf{C}}-\colon \MC \times \CM\rightarrow \mathbf{kMod}\]
is defined on objects by $(G,F)\mapsto G\otimes_{\C} F$. Given two natural transformations 
\[\Theta \in \mr{Hom}_{\MC}\left(G, G_1\right)\quad \text{and} \quad \Psi\in \mr{Hom}_{\CM}\left(F,F_1\right),\] the morphism $\Theta \otimes_{\C} \Psi$ is determined by $\left[x\otimes y\right] \mapsto \left[ \Theta_C(x)\otimes \Psi_C(y)\right]$.
\end{defn} 

It is well-known that the bifunctor $-\otimes_{\C} -$ is right exact with respect to both variables and preserves direct sums, see for example \cite[Section 1.6]{Pir02}.

\begin{defn}
We denote the left derived functors of the bifunctor $-\otimes_{\mathbf{C}}-$ by $\mathrm{Tor}_{\star}^{\mathbf{C}}(-,-)$.
\end{defn}

\subsection{Gabriel-Zisman complex}
Recall the nerve of $\C$ \cite[B.12]{Lod}. $N_{\star}\C$ is the simplicial set such that $N_n\C$ for $n\geqslant 1$ consists of all strings of composable morphisms of length $n$ in $\C$ and $N_0\C$ is the set of objects in $\C$. The face maps are defined to either compose adjacent morphisms in the string or truncate the string and the degeneracy maps insert identity morphisms into the string. We will denote an element of $N_n\C$ by $\left(f_n , \dotsc , f_1\right)$ where $f_i\in \mr{Hom}_{\C}\left(C_{i-1}, C_i\right)$.

For a small category $\C$ and a functor $F\in \CM$ there is a simplicial $k$-module, denoted $C_{\star}\left(\C , F\right)$ due to Gabriel and Zisman \cite[Appendix 2]{GZ} whose $n^{th}$ homology is canonically isomorphic to $\mr{Tor}_{n}^{\C}\left(k^{\star} , F\right)$.

\begin{defn}
\label{GabZisCpx}
Let $F\in \CM$. We define
\[C_n(\C,F)=\bigoplus_{(f_n,\dotsc ,f_1)} F(C_0)\]
where the sum runs through all elements $(f_n,\dotsc ,f_1)$ of $N_n\C$ and $f_i\in \mr{Hom}_{\C}\left(C_{i-1}, C_i\right)$. We write a generator of $C_n(\C,F)$ in the form $(f_n,\dotsc, f_1,x)$ where $(f_n,\dotsc ,f_1)\in N_n\C$ indexes the summand and $x\in F(C_0)$. 
The face maps $\p_i\colon C_n(\C,F)\rightarrow C_{n-1}(\C,F)$ are determined by
\[
\p_i(f_n,\dotsc, f_1,x)=\begin{cases}
(f_n,\dotsc ,f_2, F(f_1)(x)) & i=0,\\
(f_n, \dotsc , f_{i+1}\circ f_i,\dotsc ,f_1,x) & 1\leqslant i \leqslant n-1,\\
(f_{n-1},\dotsc ,f_1, x) & i=n.
\end{cases}\]
The degeneracy maps insert identity maps into the string.
By abuse of notation we will also denote the associated chain complex by $C_{\star}\left(\C , F\right)$.
\end{defn}

\section{Harrison Homology}
\label{Harr-section}
\subsection{Harrison homology}
We recall the definition of the Harrison complex for a commutative algebra as a quotient of the Hochschild complex. We recall that when the ground ring contains $\mathbb{Q}$ there is an alternative description of the Harrison complex as a subcomplex of the Hochschild complex.

\begin{defn}
Let $A$ be a commutative $k$-algebra. The \emph{Harrison complex} of $A$ is defined to be the quotient of the Hochschild complex $C_{\star}(A)$ by the shuffle complex $Sh_{\star}(A)$. That is,
\[CHarr_{\star}(A)\coloneqq \frac{C_{\star}(A)}{Sh_{\star}(A)}\]
with boundary map induced from the Hochschild boundary map.
\end{defn}

\begin{defn}
Let $A$ be a commutative $k$-algebra. Let $M$ be a symmetric $A$-bimodule. We define the \emph{Harrison complex of $A$ with coefficients in $M$} to be the complex
\[CHarr_{\star}(A,M) \coloneqq M\otimes_A CHarr_{\star}(A).\]
\end{defn}

\begin{defn}
The homology of the chain complex $M\otimes_A CHarr_{\star}(A)$ is called the \emph{Harrison homology of $A$ with coefficients in $M$} and is denoted by $Harr_{\star}(A,M)$.
\end{defn}

Recall the decomposition of the Hochschild complex from Theorem \ref{Hodge}. The following result is due to Barr \cite{Barr}.

\begin{prop}
Let $k\supseteq \mathbb{Q}$. Let $A$ be a commutative $k$-algebra and let $M$ be an $A$-bimodule. There is a natural isomorphism of chain complexes
\[\pushQED{\qed} e_{\star}^{(1)}C_{\star}(A,M)\cong CHarr_{\star}(A,M). \qedhere \popQED\]
\end{prop}

\subsection{Normalized Harrison homology}
\label{norm-Harr-chapter}
Suppose $k$ contains $\mathbb{Q}$ and let $A$ be an augmented commutative $k$-algebra with augmentation ideal $I$. Let $M$ be a symmetric $A$-bimodule which is flat over $k$. In this section we define a normalized version of Harrison homology under these conditions. 

Forming a normalized Harrison homology is not straightforward in general. Whilst we can form a description of the Harrison complex as a subcomplex of the Hochschild complex, this is because the Hochschild boundary map is compatible with the Eulerian idempotents. In fact, the Eulerian idempotents are not compatible with individual face and degeneracy maps, so the Harrison subcomplex does not arise as a chain complex associated to a simplicial object.

Let $D_{\star}\left(A,M\right)$ denote the degenerate subcomplex of the Hochschild complex. That is, the subcomplex of $C_{\star}(A,M)$ generated by elementary tensors for which at least one tensor factor is equal to $1_A$, the multiplicative identity in $A$. The Eulerian idempotents split the degenerate subcomplex in precisely the same way that they split the Hochschild complex. We therefore have a natural isomorphism of chain complexes
\[D_{\star}(A,M)\cong\bigoplus_{i=1}^{\infty} e_{\star}^{(i)}D_{\star}(A,M).\]
Furthermore, since the degenerate subcomplex is acyclic, each summand of the decomposition is acyclic.

By inclusion, $e_{\star}^{(i)}D_{\star}(A,M)$ is a subcomplex of $e_{\star}^{(i)}C_{\star}(A,M)$. In particular, taking $i=1$, $e_{\star}^{(1)}D_{\star}(A,M)$ is a subcomplex of the Harrison complex. This gives rise to a short exact sequence of chain complexes; for each $i\geqslant 1$ we have 
\[0\rightarrow e_{\star}^{(i)}D_{\star}(A,M)\rightarrow e_{\star}^{(i)}C_{\star}(A,M)\xrightarrow{Q_i} \frac{e_{\star}^{(i)}C_{\star}(A,M)}{e_{\star}^{(i)}D_{\star}(A,M)}\rightarrow 0.\]

By a standard construction this gives rise to a long exact sequence in homology. Since the complex $e_{\star}^{(i)}D_{\star}(A,M)$ is acyclic we can deduce that the quotient map $Q_i$ is a quasi-isomorphism for each $i\geqslant 1$.

\begin{prop}
There is an isomorphism of $k$-modules
\[H_{n}\left(e_{\star}^{(i)}C_{\star}(A,M)\right)\cong H_{n}\left(\frac{e_{\star}^{(i)}C_{\star}(A,M)}{e_{\star}^{(i)}D_{\star}(A,M)}\right)\]
for each $n\geqslant 0$ and $i\geqslant 1$. \QEDA
\end{prop}

Recall the chain complex $C_{\star}\left(I,M\right)$ of Proposition \ref{AugSplit}. Once again we can use the Eulerian idempotents to obtain a splitting: 
\[C_{\star}(I,M)\cong\bigoplus_{i=1}^{\infty} e_{\star}^{(i)}C_{\star}(I,M).\]

Having formed these two splittings of chain complexes using well-known methods, we prove that the chain complex $e_{\star}^{(i)}C_{\star}(I,M)$ computes the Harrison homology of an augmented algebra.

\begin{lem}
For each $i\geqslant 1$ there is an isomorphism of chain complexes
\[f_i\colon\frac{e_{\star}^{(i)}C_{\star}(A,M)}{e_{\star}^{(i)}D_{\star}(A,M)}\rightarrow e_{\star}^{(i)}C_{\star}(I,M).\]
\end{lem}
\begin{proof}
We can choose representatives such that the quotient complex, in degree $n$, is generated by equivalence classes
\[ \left[ e_{\star}^{(i)}(m\otimes y_1\otimes \cdots \otimes y_n)\right]\]
where each $y_i$ is an element of the augmentation ideal $I$. 

With this choice of representatives, we have a well-defined map of chain complexes 
\[f_i\colon\frac{e_{\star}^{(i)}C_{\star}(A,M)}{e_{\star}^{(i)}D_{\star}(A,M)}\rightarrow e_{\star}^{(i)}C_{\star}(I,M)\]
determined by $\left[ e_{\star}^{(i)}(m\otimes y_1\otimes \cdots \otimes y_n)\right]\mapsto  e_{\star}^{(i)}(m\otimes y_1\otimes \cdots \otimes y_n)$
on generators in degree $n$.

The inverse is given by the map determined by sending a generator $e_{\star}^{(i)}(m\otimes y_1\otimes \cdots \otimes y_n)$ of $e_{\star}^{(i)}C_{\star}(I,M)$ to its equivalence class in the quotient.
\end{proof}

Let $I_i \colon e_{\star}^{(i)}C_{\star}(I,M)\rightarrow e_{\star}^{(i)}C_{\star}(A,M)$ denote
the inclusion of chain complexes.

\begin{thm}
\label{Harr-Thm}
Let $k\supseteq\mathbb{Q}$ and let $A$ be a augmented, commutative $k$-algebra with augmentation ideal $I$. Let $M$ be a symmetric $A$-bimodule which is flat over $k$. For each $i\geqslant 1$ there is an isomorphism of chain complexes
\[e_{\star}^{(i)}C_{\star}(A,M)\cong e_{\star}^{(i)}C_{\star}(I,M)\oplus e_{\star}^{(i)}D_{\star}(A,M).\]
\end{thm}
\begin{proof}
One easily checks that the composite $f_i\circ Q_i\circ I_i$ is the identity map on the chain complex
$e_{\star}^{(i)}C_{\star}(I,M)$. The theorem then follows upon observing that $\mr{Ker}(f_i\circ Q_i) \cong e_{\star}^{(i)}D_{\star}(A,M)$.
\end{proof}

\begin{cor}
\label{HarrCor}
Taking $i=1$ in Theorem \ref{Harr-Thm}, there is an isomorphism of $k$-modules
\[ H_{n}\left(e_{\star}^{(1)}C_{\star}(I,M)\right)\cong Harr_{n}\left(A,M\right)\]
for each $n\geqslant 0$. \QEDA
\end{cor}

\begin{defn}
Let $k\supseteq \mathbb{Q}$. Let $A$ be an augmented, commutative $k$-algebra with augmentation ideal $I$ and let $M$ be a symmetric $A$-bimodule which is flat over $k$. We call the chain complex $e_{\star}^{(1)}C_{\star}(I,M)$ the \emph{normalized Harrison complex}.
\end{defn}

\section{Gamma Homology}
\label{gamma-section}
We recall the definition of $\G$-homology for a commutative $k$-algebra in the functor homology setting. We recall the Robinson-Whitehouse complex, whose homology is $\G$-homology. In the case of an augmented, commutative $k$-algebra we define a splitting of the Robinson-Whitehouse complex by defining the \emph{pruning map}. Furthermore, we utilize the normalized Harrison homology of Subsection \ref{norm-Harr-chapter} to prove that when $k\supseteq \mathbb{Q}$ and $A$ is a flat, augmented, commutative $k$-algebra we can compute $\G$-homology using only the augmentation ideal. 

\subsection{The categories $\G$ and $\Omega$}
\begin{defn}
Let $\Gamma$ denote the category whose objects are the finite based sets $[n]=\lbrace \ul{0}, 1, \dotsc ,n\rbrace$, for $n\geqslant 0$, where $\ul{0}$ denotes that the set is based at $0$. The morphisms are basepoint-preserving maps of sets.
\end{defn}

\begin{defn}
Let $\Omega$ denote the category whose objects are the finite sets $\ul{n}=\lbrace 1,\dotsc ,n\rbrace$ for $n\geqslant 1$ and whose morphisms are surjections of sets.
\end{defn}

\subsection{Gamma homology as functor homology}
\begin{defn}
Let $t$ denote the right $\Gamma$-module $\mr{Hom}_{\mathbf{Set_{\star}}}\left(-,k\right)$, where the commutative ring $k$ is considered to be a based set with basepoint $0$. The functor $t$ is sometimes referred to as the \emph{based $k$-cochain functor} \cite[Section 3.4]{RobEIO}.
\end{defn}

\begin{defn}
Let $F$ be a left $\G$-module. We define the \emph{$\G$-homology of $F$} by
\[ H\G_{\star}(F)\coloneqq\mr{Tor}_{\star}^{\G}\left(t,F\right).\]
\end{defn}

\begin{defn}
\label{gamma-loday-defn}
Let $A$ be a commutative $k$-algebra and let $M$ be a symmetric $A$-bimodule. We define the \emph{Loday functor} $\lodam (-)\colon \G \rightarrow \kmod$ on objects by
\[\lodam \left([n]\right)=A^{\otimes n}\otimes M.\]
For an element $f\in \mr{Hom}_{\G}\left([p],[q]\right)$, $\lodam(f)$ is determined by
\[\left(a_1\otimes\cdots\otimes a_p\otimes m\right)\mapsto \left(\prod_{i\in f^{-1}(1)} a_i\right) \otimes \cdots \otimes \left(\prod_{i\in f^{-1}(q)}a_i\right) \otimes \left(\prod_{i\in f^{-1}(0)}a_i\right)m\]
where an empty product is understood to be $1_A\in A$.
\end{defn}

\begin{defn}
Let $A$ be a commutative $k$-algebra and let $M$ be a symmetric $A$-bimodule. We define the \emph{$\G$-homology of $A$ with coefficients in $M$} by
\[H\G_{\star}\left(A,M\right)\coloneqq H\G_{\star}\left(\lodam\right).\]
\end{defn}

\subsection{The Robinson-Whitehouse complex}
The Robinson-Whitehouse complex for commutative algebras was first defined in the thesis of Sarah Whitehouse \cite[Definition \RNum{2}.4.1]{Whi94}. Pirashvili and Richter provided a more general construction for all left $\G$-modules \cite[Section 2]{Pir00}. 

\begin{defn}
Let $N\Omega_n\left(\ul{x},\ul{1}\right)$ denote the set of strings of composable morphisms of length $n$ in the category $\Omega$ whose initial domain is the set $\ul{x}$ and whose final codomain is the set $\ul{1}$. An element
\[ \ul{x}\xra{f_1}\ul{x_1}\xra{f_2}\cdots \xra{f_{n-1}}\ul{x_{n-1}}\xra{f_n} \ul{1}\]
of $N\Omega_n\left(\ul{x},\ul{1}\right)$ will be denoted $\left[ f_n\mid\cdots \mid f_1\right]$.
\end{defn}

\begin{rem}
Observe that a morphism $f\in \mr{Hom}_{\Omega}\left(\ul{x},\ul{x_1}\right)$ induces a map $f_{\star}\colon A^{\otimes x}\rightarrow A^{\otimes x_1}$ defined by permuting and multiplying the tensor factors.
\end{rem}

\begin{defn}
\label{i-comp-defn}
Let $\left[f_n\mid \cdots \mid f_1\right]$ be an element of $N\Omega_n\left(\ul{x},\ul{1}\right)$. We define the \emph{$i^{th}$ component} of $\left[f_n\mid \cdots \mid f_1\right]$, denoted $\left[f_{n-1}^i\mid \cdots \mid f_1^i\right]$ to be the string of morphisms corresponding to the preimage of $i\in \ul{x_{n-1}}$, re-indexed such that the domain of the $j^{th}$ morphism is the set $\ul{c}$ where $c$ is the cardinality of the set $f_j^{-1}\dotsc f_{n-1}^{-1}(i)$ and the final codomain is $\ul{1}$.
\end{defn}

\begin{defn}
Let $A$ be a commutative $k$-algebra and let $M$ be a symmetric $A$-bimodule. We define the simplicial $k$-module $C\G_{\star}(A,M)$ as follows. Let $C\G_0(A,M)=A\otimes M$ and let
\[C\G_n(A,M)\coloneqq \bigoplus_{x\geqslant 1} k\left[N\Omega_n\left( \ul{x},\ul{1}\right)\right] \otimes A^{\otimes x} \otimes M.\]

Let $X=\left[f_n\mid \cdots \mid f_1\right] \otimes \left(a_1\otimes \cdots \otimes a_n\right)\otimes m$ be a generator of $C\G_n(A,M)$.

We define the face maps $\p_i\colon C\G_n(A,M)\rightarrow C\G_{n-1}(A,M)$ to be determined by
\[\p_0\left(X\right)= \left[f_n\mid \cdots \mid f_2\right] \otimes f_{1\star}\left(a_1\otimes \cdots \otimes a_n\right)\otimes m,\]
\[\p_i\left(X\right)= \left[f_n\mid \cdots \mid f_{i+1}\circ f_i\mid \cdots \mid f_1\right] \otimes \left(a_1\otimes \cdots \otimes a_n\right)\otimes m\]
 for $1\leqslant i\leqslant n-1$ and
\[\p_n\left(X\right)= \sum_{i\in \ul{x_{n-1}}} \left[f_{n-1}^i\mid \cdots \mid f_1^i\right] \otimes \left(a_{i_1}\otimes \cdots \otimes a_{i_k}\right) \otimes\left( \prod_{j\not\in f_1^{-1}\cdots f_{n-1}^{-1}(i)} a_j\right)m, \]
where $\llb i_1,\dotsc , i_k\rrb$ is the ordered preimage of $i\in \ul{x_{n-1}}$ under $f_{n-1}\circ\cdots \circ f_1$.
The degeneracy maps insert identity morphisms into the string. By abuse of notation we will also denote the associated chain complex by $C\G_{\star}(A,M)$.
\end{defn}

\begin{rem}
In the case where $A$ is an augmented, commutative $k$-algebra, $C\G_n(A,M)$ is generated by elements of the form $\left[f_n\mid \cdots \mid f_1\right] \otimes \left(a_1\otimes \cdots \otimes a_n\right)\otimes m$ where $a_1\otimes \cdots \otimes a_n$ is a basic tensor in the sense of Definition \ref{basic-tens-defn}.
\end{rem}

\begin{defn}
\label{gamma-aug-cpx}
Let $A$ be an augmented $k$-algebra with augmentation ideal $I$ and let $M$ be a symmetric $A$-bimodule. We define the chain complex $C\G_{\star}(I,M)$ in degree $n$ by
\[C\G_n(I,M)\coloneqq \bigoplus_{x\geqslant 1} k\left[ N\Omega_n\left(\ul{x},\ul{1}\right)\right] \otimes I^{\otimes x}\otimes M\]
with the boundary map induced by that of $C\G_{\star}\left(A,M\right)$.
\end{defn}

\subsection{Pruning map}
Let $A$ be an augmented, commutative $k$-algebra with augmentation ideal $I$ and let $M$ be a symmetric $A$-bimodule which is flat over $k$. We will demonstrate that the chain complex $C\G_{\star}(I,M)$ is a direct summand of the chain complex $C\G_{\star}\left(A,M\right)$. We do so by providing a splitting map called the \emph{pruning map}.

\begin{defn}
Let $\left[f_n\mid\cdots \mid f_1\right] \otimes \left(a_1\otimes \cdots \otimes a_{x}\right)\otimes m$
be a generator of $C\G_n\left(A,M\right)$ such that $\left(a_1\otimes \cdots \otimes a_{x}\right)\otimes m$ is a basic tensor in the sense of Definition \ref{basic-tens-defn}.

Let $L=\llb l_1,\dotsc , l_h\rrb$ be the set such that $a_i\in I$ if and only if $i\in L$. 
Let
\[m_i\coloneqq \mr{Im}\left( \left(f_i\circ\cdots \circ f_1\right){\big|}_L\right).\]
Let $\widetilde{f_1}\colon \ul{\llv L\rrv}\rightarrow \ul{\llv m_1 \rrv}$
denote the map obtained from $f_1$ by restricting the domain to the set $L$, restricting the codomain to $m_1$ and re-indexing both domain and codomain in the canonical way.

For $2\leqslant i\leqslant n$ let $\widetilde{f_i}\colon \ul{\llv m_{i-1}\rrv}\rightarrow \ul{\llv m_i \rrv}$
denote the map obtained from $f_i$ by restricting the domain to the set $m_{i-1}$, restricting the codomain to the set $m_i$ and re-indexing both domain and codomain in the canonical way.
\end{defn}

\begin{rem}
Intuitively, the pruning map removes the trivial tensor factors from the basic tensor $\left(a_1\otimes \cdots \otimes a_x\right)\in A^{\otimes x}$ and prunes the graph in order to preserve the permutations and multiplications of the non-trivial tensor factors.
\end{rem}

\begin{defn}
\label{prune-defn}
Let $A$ be an augmented, commutative $k$-algebra with augmentation ideal $I$. Let $M$ be a symmetric $A$-bimodule which is flat over $k$. We define the \emph{pruning map}
\[P_{\star}\colon C\G_{\star}\left(A,M\right)\rightarrow C\G_{\star}(I,M)\]
to be the $k$-linear map of chain complexes determined in degree $n$ by
\[\left[f_n\mid\cdots \mid f_1\right] \otimes \left(a_1\otimes \cdots \otimes a_{x}\right)\otimes m \mapsto \left[\widetilde{f_n}\mid\cdots \mid \widetilde{f_1}\right] \otimes \left(a_{l_1}\otimes \cdots \otimes a_{l_h}\right)\otimes m.\]
\end{defn}

Further details on the pruning map including examples and a proof that it is well-defined map of chain complexes can be found in the author's thesis \cite[18.2, App. A]{DMG}.

\begin{thm}
\label{prune-split}
Let $A$ be an augmented, commutative $k$-algebra with augmentation ideal $I$. Let $M$ be a symmetric $A$-bimodule which is flat over $k$. Let 
\[i\colon C\G_{\star}\left(I,M\right)\rightarrow C\G_{\star}\left(A,M\right)\] denote the inclusion of the subcomplex. The composite
\[P_{\star}\circ i\colon C\G_{\star}(I,M)\rightarrow  C\G_{\star}(I,M)\]
is the identity map.
\end{thm}
\begin{proof}
An element in the image of $i$ is a $k$-linear combination of generators of the form 
\[\left[f_n\mid\cdots \mid f_1\right] \otimes \left(y_1\otimes \cdots \otimes y_{x}\right)\otimes m\]
such that $\left(y_1\otimes \cdots \otimes y_{x}\right) \in I^{\otimes x}$. In particular, $\left(y_1\otimes \cdots \otimes y_{x}\right)$ contains no trivial factors. By construction, the pruning map $P_{\star}$ is the identity on such elements.
\end{proof}

\begin{cor}
Under the conditions of Theorem \ref{prune-split}, there is an isomorphism of chain complexes
\[C\G_{\star}\left(A,M\right)\cong C\G_{\star}(I,M)\oplus \mr{Ker}\left(P_{\star}\right) \] 
and hence there is an isomorphism of $k$-modules
\[H\G_{n}\left(A,M\right)\cong H\G_{n}(I,M)\oplus H_{n}\left(\mr{Ker}(P_{\star})\right)\]
for each $n\geqslant 0$. \QEDA
\end{cor}

\begin{thm}
\label{gamma-iso-thm}
Let $k\supseteq \mathbb{Q}$. Let $A$ be a flat, augmented, commutative $k$-algebra with augmentation ideal $I$ and let $M$ be a symmetric $A$-bimodule. Then there is an isomorphism of graded $k$-modules
\[H\G_{\star}\left(I, M\right) \cong H\G_{\star}\left( A , M\right).\]
\end{thm}
\begin{proof}
Whitehouse \cite[Theorem \RNum{3}.4.2]{Whi94} proves that under these conditions there is an isomorphism
\[H\G_{n-1}\left(A,M\right) \cong Harr_n\left(A,M\right)\]
for $n\geqslant 1$. One may check that the same method yields an isomorphism between the $n^{th}$ homology of the complex $e_{\star}^{(1)}C_{\star}(I,M)$ and $H\G_{n-1}\left(I,M\right)$. The theorem then follows from Corollary \ref{HarrCor}.
\end{proof}

\section{Symmetric Homology}
\label{symm-hom-section}
We recall the definition of the category $\DS$ and the symmetric bar construction. We recall the definition of reduced symmetric homology and the chain complex that computes it. The material in this section can be found in \cite{Ault}.

\subsection{The category $\DS$}
In order to best define the comparison map we provide a description of the category $\DS$ such that it has the same objects as the category $\Omega$. Our definition is isomorphic to the definition found in \cite{FL} and \cite{Ault}, the isomorphism being given by shifting the index.

\begin{defn}
The category $\Delta$ has as objects the sets $\ul{n}=\llb 1,\dotsc , n\rrb$ for $n\geqslant 1$ and order-preserving maps as morphisms.
\end{defn}

\begin{rem}
Morphisms in $\Delta$ are generated by the \emph{face maps} $\delta_i \in \mr{Hom}_{\Delta}\left(\ul{n},\ul{n+1}\right)$ and the \emph{degeneracy maps} $\sigma_j\in \mr{Hom}_{\Delta}\left(\ul{n+1}, \ul{n}\right)$ for $n\geqslant 1$, $1\leqslant i\leqslant n+1$ and $1\leqslant j\leqslant n$ subject to the usual relations, see for example \cite[Appendix B]{Lod}.
\end{rem}

\begin{defn}
The category $\DS$ has as objects the sets $\ul{n}=\llb 1,\dotsc , n\rrb$ for $n\geqslant 1$. An element of $\mr{Hom}_{\DS}\left(\ul{n},\ul{m}\right)$ is a pair $\left(\phi , g\right)$ where $g \in \Sigma_n$ and $\phi\in \mr{Hom}_{\Delta}\left(\ul{n},\ul{m}\right)$.

For $\left(\phi,g\right)\in \mr{Hom}_{\DS}\left(\ul{n},\ul{m}\right)$ and $\left(\psi , h\right) \in \mr{Hom}_{\DS}\left(\ul{m},\ul{l}\right)$ the composite is the pair 
\[\left(\psi\circ h_{\star}(\phi), \phi^{\star}(h)\circ g\right)\]
where $h_{\star}(\phi)$ and $\phi^{\star}(h)$ are determined by the relations $h_{\star}(\delta_i) =\delta_{h(i)}$, $h_{\star}(\sigma_j) =\sigma_{h(j)}$ and
\begin{alignat*}{2}
    & \begin{aligned} &  \delta_i^{\star}\left(\theta_k\right)=\begin{cases}
  \theta_k & k< i-1\\
  id_{\ul{n-1}} & k=i-1\\
  id_{\ul{n-1}} & k=i\\
  \theta_{k-1} & k>i\\
  \end{cases}\\
  \end{aligned}
    & \hskip 6em &
  \begin{aligned}
  & \sigma_j^{\star}(\theta_k)= \begin{cases}
  \theta_k & k<j-1\\
  \theta_j\theta_{j-1} & k=j-1\\
  \theta_j\theta_{j+1} & k=j\\
  \theta_{k+1} & k>j\\
  \end{cases} 
  \end{aligned}
\end{alignat*}
where the $\delta_i$ and $\sigma_j$ are the face and degeneracy maps of the category $\Delta$ and the $\theta_k$ are transpositions.
\end{defn}

\begin{rem}
The category $\DS$ is isomorphic to the category of non-commutative sets as found in \cite[A10]{FT} and \cite[Section 1.2]{Pir02}.
\end{rem}

\begin{rem}
A morphism $\left( \phi , g\right)$ in $\DS$ is an epimorphism if and only if the morphism $\phi$ in $\Delta$ is an epimorphism.
\end{rem}

\begin{defn}
Let $\EDS$ denote the subcategory of $\DS$ whose objects are the sets $\ul{n}=\llb 1\dotsc , n\rrb$ for $n\geqslant 1$ and whose morphisms are the epimorphisms of $\DS$.
\end{defn}

\subsection{Symmetric homology}
\label{symm-bar-section}

\begin{defn}
Let $A$ be an associative $k$-algebra. We define the \emph{symmetric bar construction}
\[\SB\colon \DS\rightarrow \kmod\]
to be the functor defined on objects by $\SB\left(\ul{n}\right)=A^{\otimes n}$. 
For a morphism $\left(\phi , g\right)\in \mr{Hom}_{\DS}\left(\ul{n},\ul{m}\right)$, $\SB\left(\phi , g\right)$ is determined by 
\[a_1\otimes \cdots \otimes a_n \mapsto \left( \underset{_{i\in \varphi^{-1}(1)}}{{\prod}^{<}} a_{g^{-1}(i)}\right)\otimes \cdots \otimes \left(\underset{_{i\in \varphi^{-1}(m)}}{{\prod}^{<}} a_{g^{-1}(i)} \right)\]
where the product is ordered according to the morphism $\phi$. An empty product is defined to be the multiplicative unit $1_A$.
\end{defn}

\begin{defn}
Let $A$ be an augmented, associative $k$-algebra with augmentation ideal $I$. We define the functor
\[\SBI\colon \EDS\rightarrow \kmod\]
to be defined on objects by $\SBI\left(\ul{n}\right) = I^{\otimes n}$.
For a morphism $\left(\phi , g\right)\in \mr{Hom}_{\DS}\left(\ul{n},\ul{m}\right)$ we define $\SBI\left(\phi , g\right)= \SB\left(\phi , g\right)$.
\end{defn}

\begin{defn}
Let $A$ be an associative $k$-algebra. For $n\geqslant 0$, we define the \emph{$n^{th}$ symmetric homology of $A$} to be
\[HS_{n}(A)\coloneqq \mr{Tor}_{n}^{\DS}\left(k^{\star}, \SB\right),\]
where $k^{\star}$ is the trivial right $\DS$-module of Definition \ref{triv-c-mod}.
\end{defn}

By \cite[Section 6]{Ault} we can make the following definition.

\begin{defn}
\label{reduced-symm}
Let $A$ be an augmented associative $k$-algebra with augmentation ideal $I$. We define the \emph{reduced symmetric homology of $A$}, denoted $\widetilde{HS}_{\star}(A)$ to be the homology of the chain complex $CS_{\star}(I)\coloneqq C_{\star}\left(\EDS , \SBI\right)$.
\end{defn}

\begin{rem}
By \cite[Section 6.3]{Ault}, the reduced symmetric homology of the augmented $k$-algebra $A$ coincides with the symmetric homology of $A$ in non-negative degrees. Furthermore,
\[HS_0(A)\cong \widetilde{HS}_0(A) \oplus k.\]
\end{rem}

\section{Comparison Map}
\label{CM-section}
Let $A$ be an augmented, commutative $k$-algebra with augmentation ideal $I$. We define a surjective map of chain complexes
\[NCS_{\star}(I)\rightarrow NC\Gamma_{\star}(I,k)\]
between the normalized chain complex of $CS_{\star}(I)$ from Definition \ref{reduced-symm} and the normalized chain complex of $C\Gamma_{\star}(I,k)$ from Definition \ref{gamma-aug-cpx}. We therefore obtain a long exact sequence in homology connecting the reduced symmetric homology of $A$ with a summand of the $\G$-homology of $A$. By Theorem \ref{gamma-iso-thm} we obtain a long exact sequence in homology connecting the reduced symmetric homology with the entire $\G$-homology when $k\supseteq \mathbb{Q}$ and $A$ is flat over $k$.

\subsection{A quotient of the symmetric complex}
Recall the chain complex $CS_{\star}(I)$ of Definition \ref{reduced-symm}. This is a chain complex associated to a simplicial $k$-module so we can form the normalized chain complex $NCS_{\star}(I)$. In degree $n$ we have the $k$-module generated by equivalence classes of the form $\left[ (f_n,\dotsc , f_1),(y_1\otimes \cdots \otimes y_{x})\right]$ where
\[ \ul{x}\xra{f_1} \cdots  \xra{f_n} \ul{x_n}\]
is a string of non-identity morphisms $N_n(\EDS)$ and $y_1\otimes\cdots \otimes y_{x}\in \SBI\left(\ul{x}\right)$. The boundary map is given by the alternating sum of face maps described in Definition \ref{GabZisCpx}.

\begin{defn}
\label{quot-cpx-defn}
Denote by $NCS_{\star}^1(I)$ the subcomplex of $NCS_{\star}(I)$ generated by the equivalence classes for which $\ul{x_n}\neq \ul{1}$. Let
\[q\colon NCS_{\star}(I)\rightarrow \ol{NCS_{\star}(I)} \coloneqq \frac{NCS_{\star}(I)}{NCS_{\star}^1(I)}\]
denote the quotient map.
\end{defn}

\begin{rem}
Observe that this well-defined. The only face map that affects the final codomain in the string is the last face map, which omits the last morphism in the string. Since this last morphism was a non-identity surjection and $x_n>1$, we see $x_{n-1}>1$.
\end{rem}

The $k$-module $\ol{NCS_n(I)}$ is generated by equivalence classes $\left[ (f_n,\dotsc ,f_1),(y_1\otimes\cdots \otimes y_{x})\right]$
where $(f_n,\dotsc , f_1)$ denotes a string of non-identity morphisms 
\[ \ul{x}\xra{f_1}\cdots \xra{f_{n-1}}\ul{x_{n-1}}\xra{f_n} \ul{1}\]
in $N_n(\EDS)$ and $(y_0\otimes\cdots \otimes y_{x})\in \SBI(\ul{x})$.

\subsection{Mapping to the Robinson-Whitehouse complex}
Recall the chain complex $C\G_{\star}(I,k)$ from Definition \ref{gamma-aug-cpx}. Since $C\G_{\star}(I,k)$ is the chain complex associated to a simplicial $k$-module we can take the normalized complex, $NC\G_{\star}(I,k)$. A generator of the $k$-module $NC\G_n(I,k)$ is an equivalence class of the form
\[\left[\left[f_n\mid f_{n-1}\mid \cdots \mid f_1\right] \otimes (y_1\otimes\cdots \otimes y_{x})\otimes 1_k\right],\]
where $\left[f_n\mid f_{n-1}\mid \cdots \mid f_1\right]$ denotes a string of non-identity morphisms 
\[ \ul{x}\xra{f_1}\cdots \xra{f_{n-1}}\ul{x_{n-1}}\xra{f_n} \ul{1}\]
in $N_n\Omega$ and $(y_1\otimes\cdots \otimes y_{x})\in I^{\otimes x}$.

The category $\DS$ is isomorphic to the category of non-commutative sets, $\fas$. We can therefore consider a morphism in $\EDS$ to be a surjection of finite sets, that is a morphism in $\Omega$, with the extra structure of a total ordering on the preimage of each singleton in the codomain.

\begin{defn}
Let $U\colon \EDS\rightarrow \Omega$ denote the forgetful functor which is the identity on objects and sends a morphism of $\EDS$ to the underlying surjection of sets.
\end{defn}

\begin{defn}
Let 
\[\Phi_n\colon\ol{NCS_n(I)}\rightarrow NC\G_n(I,k)\]
be the map of $k$-modules determined by
\[ \left[ (f_n,\dotsc , f_1),(y_1\otimes\cdots \otimes y_{x})\right]\mapsto \left[\left[ U(f_n)\mid \cdots \mid U(f_1)\right] \otimes (y_1\otimes\cdots \otimes y_{x})\otimes 1_k\right].\]
\end{defn}

The proof that the maps $\Phi_n$ assemble into a map of chain complexes can be found in \cite[Appendix D]{DMG}. Since both chain complexes arise from simplicial $k$-modules the proof consists of a standard but time-consuming check that the maps $\Phi_n$ are compatible with the face maps.

\subsection{Main result}
\begin{thm}
Let $A$ be an augmented, commutative $k$-algebra with augmentation ideal $I$. There is a surjective map of chain complexes
\[\Phi\circ q\colon NCS_{\star}(I)\rightarrow NC\G_{\star}(I,k).\]
\end{thm}
\begin{proof}
Recall from Definition \ref{quot-cpx-defn} that $q\colon NCS_{\star}(I)\rightarrow \ol{NCS_{\star}(I)}$ is a quotient map and is therefore surjective.

A generator
\[\left[\left[f_n\mid f_{n-1}\mid \cdots \mid f_1\right] \otimes (y_1\otimes\cdots \otimes y_{x})\otimes 1_k\right]\]
of $NC\G_{n}(I,k)$ is the image of 
\[ \left[ (f_n,\dotsc ,f_1),(y_1\otimes\cdots \otimes y_{x})\right] \]
in $\ol{NCS_{n}(I)}$ where we take the total orderings on the preimages of each $f_i$ to be the canonical ones. Hence $\Phi$ is also a surjective map.
\end{proof}

\begin{cor}
\label{LES-cor}
Let $A$ be an augmented, commutative $k$-algebra with augmentation ideal $I$. There is a short exact sequence of chain complexes
\[0\rightarrow \mr{Ker}(\Phi\circ q)\rightarrow NCS_{\star}(I)\xrightarrow{\Phi\circ q} NC\G_{\star}(I,k)\rightarrow 0,\]
which gives rise to the long exact sequence

\begin{center}
\begin{tikzcd}
  \cdots \rar   & H_n(\mr{Ker}(\Phi\circ q)\rar &  \widetilde{HS}_n\left(A\right) \rar & H\G_n(I,k)
          \ar[out=0, in=180, looseness=2, overlay]{dll}   & \\
        &  H_{n-1}(\mr{Ker}(\Phi\circ q)) \rar & \cdots \rar & H\G_1(I,k) \ar[out=0, in=180, looseness=2, overlay]{dll} & \\ & H_0(\mr{Ker}(\Phi\circ q)) \rar & \widetilde{HS}_0\left(A\right) \rar & H\G_0(I,k) \rar & 0
         \end{tikzcd}
\end{center} 
connecting the reduced symmetric homology of $A$ with a direct summand of the $\G$-homology of $A$. \QEDA
\end{cor}

\begin{thm}
\label{long-exact-thm}
Let $k\supseteq\mathbb{Q}$. Let $A$ be a flat, augmented, commutative $k$-algebra. There is a long exact sequence
\begin{center}
\begin{tikzcd}
  \cdots \rar   & H_n(\mr{Ker}(\Phi\circ q)\rar &  \widetilde{HS}_n\left(A\right) \rar & H\G_n\left(A,k\right)
          \ar[out=0, in=180, looseness=2, overlay]{dll}   & \\
        &  H_{n-1}(\mr{Ker}(\Phi\circ q)) \rar & \cdots \rar & H\G_1\left(A,k\right) \ar[out=0, in=180, looseness=2, overlay]{dll} & \\ & H_0(\mr{Ker}(\Phi\circ q)) \rar & \widetilde{HS}_0\left(A\right) \rar & H\G_0\left(A,k\right) \rar & 0
         \end{tikzcd}
\end{center} 
connecting the reduced symmetric homology of $A$ with the $\G$-homology of $A$ with coefficients in $k$. 
\end{thm}
\begin{proof}
The theorem follows from Corollary \ref{LES-cor} and Theorem \ref{gamma-iso-thm}.
\end{proof}

\section*{Acknowledgements}
I would like to thank Sarah Whitehouse, my Ph.D. supervisor, for all her guidance and encouragement. I am grateful to James Brotherston for his proof-reading and helpful suggestions.

\bibliographystyle{alpha}
\bibliography{comp-refs}

\begin{thebibliography}{Whi94}

\bibitem[Aul10]{Ault}
Shaun~V. Ault.
\newblock Symmetric homology of algebras.
\newblock {\em Algebr. Geom. Topol.}, 10(4):2343--2408, 2010.

\bibitem[Aul14]{Ault-HO}
Shaun~V. Ault.
\newblock Homology operations in symmetric homology.
\newblock {\em Homology Homotopy Appl.}, 16(2):239--261, 2014.

\bibitem[Bar68]{Barr}
Michael Barr.
\newblock Harrison homology, {H}ochschild homology and triples.
\newblock {\em J. Algebra}, 8:314--323, 1968.

\bibitem[Fie]{Fie}
Z.~Fiedorowicz.
\newblock The symmetric bar construction.
\newblock URL: \url{https://people.math.osu.edu/fiedorowicz.1/}.

\bibitem[FL91]{FL}
Zbigniew Fiedorowicz and Jean-Louis Loday.
\newblock Crossed simplicial groups and their associated homology.
\newblock {\em Trans. Amer. Math. Soc.}, 326(1):57--87, 1991.

\bibitem[FT87]{FT}
B.~L. Fe\u{\i}gin and B.~L. Tsygan.
\newblock Additive {$K$}-theory.
\newblock In {\em {$K$}-theory, arithmetic and geometry ({M}oscow,
  1984--1986)}, volume 1289 of {\em Lecture Notes in Math.}, pages 67--209.
  Springer, Berlin, 1987.

\bibitem[Gra19]{DMG}
Daniel Graves.
\newblock {\em Functor Homology for Augmented Algebras}.
\newblock PhD thesis, University of Sheffield, 2019.

\bibitem[GS87]{GS87}
Murray Gerstenhaber and S.~D. Schack.
\newblock A {H}odge-type decomposition for commutative algebra cohomology.
\newblock {\em J. Pure Appl. Algebra}, 48(3):229--247, 1987.

\bibitem[GZ67]{GZ}
P.~Gabriel and M.~Zisman.
\newblock {\em Calculus of fractions and homotopy theory}.
\newblock Ergebnisse der Mathematik und ihrer Grenzgebiete, Band 35.
  Springer-Verlag New York, Inc., New York, 1967.

\bibitem[Lod98]{Lod}
Jean-Louis Loday.
\newblock {\em Cyclic homology}, volume 301 of {\em Grundlehren der
  Mathematischen Wissenschaften [Fundamental Principles of Mathematical
  Sciences]}.
\newblock Springer-Verlag, Berlin, second edition, 1998.
\newblock Appendix E by Mar{\'{\i}}a O. Ronco, Chapter 13 by the author in
  collaboration with Teimuraz Pirashvili.

\bibitem[LV12]{LodVal}
Jean-Louis Loday and Bruno Vallette.
\newblock {\em Algebraic operads}, volume 346 of {\em Grundlehren der
  Mathematischen Wissenschaften [Fundamental Principles of Mathematical
  Sciences]}.
\newblock Springer, Heidelberg, 2012.

\bibitem[Pir00]{PirH}
Teimuraz Pirashvili.
\newblock Hodge decomposition for higher order {H}ochschild homology.
\newblock {\em Ann. Sci. \'Ecole Norm. Sup. (4)}, 33(2):151--179, 2000.

\bibitem[PR00]{Pir00}
T.~Pirashvili and B.~Richter.
\newblock Robinson-{W}hitehouse complex and stable homotopy.
\newblock {\em Topology}, 39(3):525--530, 2000.

\bibitem[PR02]{Pir02}
T.~Pirashvili and B.~Richter.
\newblock Hochschild and cyclic homology via functor homology.
\newblock {\em $K$-Theory}, 25(1):39--49, 2002.

\bibitem[Rob18]{RobEIO}
Alan Robinson.
\newblock {$E_\infty$} obstruction theory.
\newblock {\em Homology Homotopy Appl.}, 20(1):155--184, 2018.

\bibitem[Wei94]{weib}
Charles~A. Weibel.
\newblock {\em An introduction to homological algebra}, volume~38 of {\em
  Cambridge Studies in Advanced Mathematics}.
\newblock Cambridge University Press, Cambridge, 1994.

\bibitem[Whi94]{Whi94}
Sarah~A. Whitehouse.
\newblock {\em Gamma (co)homology of commutative algebras and some related
  representations of the symmetric group}.
\newblock PhD thesis, University of Warwick, 1994.

\end{thebibliography}

\end{document}